\documentclass[a4paper,12pt]{article}
\usepackage[T2A]{fontenc}
\usepackage[cp1251]{inputenc}
\usepackage[english
]{babel}
\usepackage[tbtags]{amsmath}
\usepackage{amsfonts,amssymb,
amscd}

\def\AC{\sim}

\def \GK{\operatorname{GK}}
\def\Pid{\operatorname{Pid}}
\def\Var{\operatorname{Var}}
\def\Ht{\operatorname{Ht}}
\def\PI{{\rm PI}}

\newtheorem{theorem}{Theorem}
\newtheorem{lemma}{Lemma}[section]
\newtheorem{proposition}{Proposition}[section]

\newtheorem{corollary}{Corollary}[section]
\newtheorem{notation}{Notation}[section]
\newtheorem{remark}{Remark}[section]

\newtheorem{algorithm}{Algorithm}[section]
\newtheorem{ques}{Question}[section]

\newtheorem{definitionhead}[theorem]{Определение}
\newenvironment{definition}{\begin{definitionhead}%
\sl}{\end{definitionhead}}

\newtheorem{proof}{Доказательство}

\begin{document}\title{
\ \hbox to \textwidth{\normalsize      
16R+05E\hfill} 
\ \hbox to \textwidth{\normalsize
УДК 512.5+512.64+519.1\hfill}\\[1ex]
Subexponential estimates in Shirshov theorem on height}

\author {Alexei Belov, Mikhail Kharitonov}

\maketitle


\begin{abstract}
In 1993 E.~I.~Zelmanov has put the following question in Dniester Notebook:{\it Suppose that $F_{2, m}$ is a free 2-generated associative ring with the identity $x^m=0.$ Is it true that the nilpotency degree of $F_{2, m}$ has exponential growth?}

We give the definitive answer to E.~I.~Zelmanov's question showing that the nilpotency class of an $l$-generated associative algebra with the identity $x^d=0$ is smaller than $\Psi(d,d,l),$ where $$\Psi(n,d,l)=2^{18} l (nd)^{3 \log_3 (nd)+13}d^2.$$ This result is a consequence of the following fact  based on combinatorics of words. Let $l$, $n$ и $d \geqslant n$ be positive integers. Then all words over an alphabet of cardinality  $l$ whose length is not less than $\Psi(n,d,l)$ are either $n$-divisible or contain $d$-th power of a subword; a word  $W$ is {\it $n$-divisible} if it can be represented in the
form  $W=W_0W_1\cdots W_n$ such that $W_1,W_2,\dots,W_n$ are placed in lexicographically decreasing order. Our proof uses Dilworth theorem (according to V.~N.~Latyshev's idea). We show that the set of not $n$-divisible words over an alphabet of cardinality $l$ has height $h<\Phi(n,l)$ over the set of words of degree $\leqslant (n-1)$, where
$$\Phi(n,l) = 2^{87} l\cdot n^{12\log_3 n + 48}.$$

\end{abstract}


\medskip

{\bf Keywords:}
Shirshov theorem on height,
word combinatorics, $n$-divisibility, Dilworth theorem, Burnside-type problems.

\section{Introduction}
\subsection{Shirshov theorem on height}

In 1958 A.~I.~Shirshov has proved his famous theorem on height (\cite{Shirshov1}, \cite{Shirshov2}).

\begin{definition}
A word $W$ is called {\em $n$-divisible} if $W$ can be represented in the form
$W = vu_1u_2\cdots u_n$ such that $u_1\succ
u_2\succ\dots\succ u_n$.
\end{definition}

In this case any non-identical permutation $\sigma$ of subwords $u_i$ produces a word $W_\sigma =
vu_{\sigma(1)}u_{\sigma(2)}\cdots u_{\sigma(n)}$,
which is lexicographically smaller than $W$. Some authors take this feature as the definition of
$n$-divisibility.

\begin{definition}
A \PI-algebra $A$ is called an algebra {\em of bounded height
$h=\Ht_Y(A)$ over a set of words $Y = \{ u_1, u_2,\ldots\}$} if
$h$ is the minimal integer such that any word $x$ from $A$ can be represented in the form
$$x = \sum_i \alpha_i u_{j_{(i,1)}}^{k_{(i,1)}}
u_{j_{(i, 2)}}^{k_{(i,2)}}\cdots
u_{j_{(i,r_i)}}^{k_{(i,r_i)}}
$$
where $\{r_i\}$ don't exceed $h$. The set
$Y$ is called {\em a Shirshov basis} for $A$.
If no misunderstanding can occur, we use $h$ instead of $\Ht_Y(A)$.
\end{definition}

\medskip
 {\bf Shirshov theorem on height.} (\cite{Shirshov1}, \cite{Shirshov2})
 {\it The set of not $n$-divisible words in a finitely generated algebra with an admissible polynomial identity has bounded height $H$ over the set of words of degree not exceeding $n-1$.}
 \medskip

The Burnside-type problems related to height theorem are considered in \cite{Zelmanov}. The authors believe that Shirshov theorem on height is a fundamental fact in word combinatorics independently of its applications to $\PI$-theory. (All our proofs are elementary and fit in the framework of word combinatorics.) Unfortunately the experts in combinatorics haven't sufficiently appraised this fact yet. As regards the notion of
{\it $n$-divisibility} itself, it seems to be fundamental as well. V.~N.~Latyshev's estimates on $\xi_n(k)$, the number of non-$n$-divisible polylinear words in $k$ symbols, have lead to fundamental results in $\PI$-theory. At the same time, this number is nothing but the number of arrangements of integers from $1$ to $k$, such that no $n$ integers (not necessarily consecutive) are placed in decreasing order.
Furthermore it is the number of permutationally ordered sets of diameter $n$ consisting of $k$ elements.
(A set is called {\it permutationally ordered} if its ordering is the intersection of two linear orderings, {\it the diameter} of an ordered set is the length of its maximal antichain.)

 Height theorem implies the solution of a set of problems
 in ring theory. Suppose an associative algebra over a field satisfies a polynomial identity
$f(x_1,\ldots,x_n)=0$. It is possible to prove that then it satisfies an admissible polynomial identity (that is, a polynomial identity with coefficient $1$ at some term of the highest degree):
$$
x_1 x_2\cdots x_n = \sum_{\sigma}\alpha_{\sigma}x_{\sigma(1)}
x_{\sigma(2)}\cdots x_{\sigma(n)},$$
 where $\alpha_{\sigma}$ belong to the ground field. In this case, if
 $W = v u_1 u_2 \cdots u_n$ is $n$-divisible then for any permutation $\sigma$ the word $W_{\sigma} = vu_{\sigma(1)}u_{\sigma(2)}\cdots u_{\sigma(n)}$
is lexicographically smaller than $W$, and thus  an $n$-divisible word
can be represented as a linear combination of lexicographically smaller words. Hence a
 $\PI$-algebra has a basis consisting of non-$n$-divisible words. By Shirshov theorem on height, a $\PI$-algebra has bounded height. In particular, if a $\PI$-algebra satisfies $x^n = 0$ then it is nilpotent, that is, any its word of length exceeding some $N$ is identically zero. Surveys on height theorem can be found in \cite{BBL97,Kem09,BelovRowenShirshov, Ufn90, Dr04}.

This theorem implies the positive solution of Kurosh problem and of other Burnside-type problems
for $\PI$-колец. Indeed, if $Y$~is a Shirshov basis and all its elements are algebraic then the algebra
$A$ is finite-dimensional. Thus Shirshov theorem explicitly indicates a set of elements whose algebraicity makes the whole algebra finite-dimensional. This theorem implies

\begin{corollary}[Berele]
Let $A$~be a finitely generated $\PI$-algebra. Then
$$\GK(A)<\infty.$$
\end{corollary}

$\GK(A)$ is {\it the Gelfand --- Kirillov dimension of the algebra $A$},
that is,
$$\GK(A)=\lim_{n\to\infty}\ln V_A(n)/\ln(n)$$
where $V_A(n)$ is  {\it the growth function of $A$},
the dimension of the vector space generated by the words of degree not greater than
$n$ in the generators of $A$.

Indeed, it suffices to observe that the number of solutions for the inequality
$k_1 |v_1|+\cdots+k_h|v_h|\leqslant n$ with $h\leqslant H$ exceeds $N^{H}$, so that $\GK(A)\leqslant \Ht(A)$.

The number $m=\deg(A)$ will mean {\it the degree of the algebra}, or
the minimal degree of an identity valid in it.
The number $n=\Pid(A)$ is {\it the complexity} of $A$, or the maximal $k$
such that ${\Bbb M}_k$, the algebra of matrices of size $k$, belongs to the variety
$\Var(A)$ generated by $A$.

Instead of the notion of {\it height}, it is more suitable to use the close notion of
{\it essential height}.

\begin{definition}
An algebra $A$ has {\em essential height $h=H_{Ess}(A)$} over
a finite set $Y$ called {\em an $s$-basis for $A$} if there exists
a finite set $D\subset A$ such that $A$ is linearly representable by elements of the form
$t_1\cdot\ldots\cdot t_l$, where $l\leqslant
2h+1$, $\forall i (t_i\!\in\! D \vee t_i=y_i^{k_i};y_i\in Y)$
and the set of $i$ such that $t_i\not\in D$ contains not more than
$h$ elements. The {\em essential height} of a set of words is defined similarly.
\end{definition}

Informally speaking, any long word is a product of periodic parts and ''gaskets'' of restricted length.
The essential height is the number of periodic parts, and the ordinary height accounts ``gaskets'' as well.

The height theorem suggests the following questions:

\begin{enumerate}

\item To which classes of rings the height theorem can be extended?

\item Over which $Y$ the algebra $A$ has bounded height? In particular, what sets of words can be taken for
$\{v_i\}$?

\item What is the structure of the degree vector $(k_1,\ldots,k_h)$? First of all,
what sets of its components are essential, that is, what sets of
$k_i$ can be unbounded simultaneously?
What is the value of essential height?
Is it true
that the set of degree vectors has some regularity properties?

\item What estimates for the height are possible?
\end{enumerate}

Let us discuss the above questions.
\subsection{Non-associative generalizations}
 The height theorem was extended to some classes of near-associative rings.
S.~V~Pchelintsev \cite{Pchelintcev} has proved it for the alternative and
the $(-1,1)$ сases, S.~P.~Mishchenko \cite{Mishenko1} has obtained an analogue
of the height theorem for Lie algebras with a sparse identity. In the paper by one of the authors
\cite{Belov1} the height theorem was proved for some class of rings asymptotically close to associative rings. In particular, this class contains alternative and Jordan $\PI$-algebras.
\subsection{Shirshov bases}
Suppose $A$ is a $\PI$-algebra and a subset
$M\subseteq A$ is its $s$-basis. Then if all elements of
$M$ are algebraic over $K$ then $A$
is finite-dimensional (the Kurosh problem). Boundedness of essential height
over $Y$ implies ``the positive solution of the Kurosh problem over
$Y$''. The converse is less trivial.

\begin{theorem}  [A.~Ya.~Belov]         \label{ThKurHmg}
a) Suppose $A$~is a graded $\PI$-algebra, $Y$~is a finite set of homogeneous elements.
Then if for all $n$ the algebra
$A/Y^{(n)}$ is nilpotent then $Y$ is an $s$-basis for $A$. If moreover
$Y$ generates $A$ as an algebra then $Y$~is a Shirshov basis for $A$.

b) Suppose $A$~is a $\PI$-algebra, $M\subseteq A$~is a Kurosh subset
in $A$. Then $M$~is an $s$-basis for $A$.
\end{theorem}

Let $Y^{(n)}$ denote the ideal generated by $n$th powers of elements from
$Y$. A set $M\subset A$ is called {\it a Kurosh set} if any projection
$\pi\colon A\otimes K[X]\to A'$ such that the image
$\pi(M)$ is entire over $\pi(K[X])$ is finite-dimensional over $\pi(K[X])$.
The following example motivates this definition. Suppose
$A={\Bbb Q}[x,1/x]$. Any projection $\pi$ such that $\pi(x)$
is algebraic has a finite-dimensional image. However the set $\{x\}$
is not an $s$-basis for ${\Bbb Q}[x,1/x]$. Thus  boundedness
of essential height is a non-commutative generalization of the property of
{\it entireness}.
\subsection{Shirshov bases consisting of words}
The Shirshov bases consisting of words are described by the following

\begin{theorem}[\cite{BBL97}, \cite{BR05}]            \label{ThBelheight}
A set $Y$ of words is a Shirshov basis for an algebra $A$ iff for any word
$u$ of length not exceeding $m =
\Pid(A)$, the complexity of $A$, the set $Y$ contains a word cyclically conjugate
to some power of $u$.
\end{theorem}

A similar result was obtained independently by G.~P.~Chekanu and V.~Drensky. Problems related to local finiteness of algebras and to algebraic sets of words of degree not exceeding the complexity of the algebra were investigated in \cite{Ufn90, Cio97, Cio88, CK93, Che94, Ufn85, UC85}. Questions related to generalization of the independence theorem were considered in these papers as well.

\subsection{Essential height}
Clearly the Gelfand --- Kirillov dimension is estimated by the essential height.
Furthermore an $s$-basis is a Shirshov basis iff it generates
$A$ as an algebra. In the representable case the converse is also true.

\begin{theorem}[A.~Ya.~Belov, \cite{BBL97}]
Suppose $A$~is a finitely generated representable algebra and
$H_{Ess}{}_Y(A)<\infty$. Then $H_{Ess}{}_Y(A)=\GK(A)$.
\end{theorem}

\begin{corollary}[V.~T.~Markov]
The Gelfand --- Kirillov dimension of a finitely generated representable algebra is an integer.
\end{corollary}

\begin{corollary}
If $H_{Ess}{}_Y(A)<\infty$ and $A$ is representable then
$H_{Ess}{}_Y(A)$ is independent of choice of the $s$-basis $Y$.
\end{corollary}

In this case the Gelfand --- Kirillov dimension also is equal to the essential height by virtue of local representability of relatively free algebras.

\medskip
{\bf Structure of degree vectors.} Although in the representable case the Gelfand --- Kirillov dimension and the essential height behave well, even in this case the set of degree vectors may have a bad structure, namely, it can be the complement to the set of solutions of a system of exponential-polynomial Diophantine equations \cite{BBL97}. That is why there exists an instance of a representable algebra with the trascendent Hilbert series. However for a relatively free algebra, the Hilbert series is rational \cite{Belov501}.
\medskip
\subsection{$n$-divisibility and Dilworth theorem}
The significance of the notion of {\it
$n$-divisibility} exceeds the limits of Burnside-type problems. This notion is also actual in investigation of polylinear words and estimation of their number; a word is {\it polylinear} if each letter occurs in it at most one time. V.~N.~Latyshev applied Dilworth theorem for estimation of the number of not $m$-divisible polylinear words of degree $n$ over the alphabet
$\{a_1,\dots,a_n\}$. The estimate is ${(m - 1)}^{2n}$ and is rather sharp. Let us recall this theorem.

\medskip
{\bf Dilworth theorem}: {\it Let $n$ be the maximal number of elements of an antichain in a given finite partially ordered set $M$. Then $M$ can be divided into $n$ disjoint chains.}
\medskip

Consider a polylinear word $W$ consisting of $n$ letters. Put $a_i\succ
a_j$ if $i>j$ and the letter $a_i$ is located in $W$ to the right from $a_j$.
The condition of not $k$-divisibility means absence of an antichain consisting of $n$
elements. Then by Dilworth theorem all positions (and the letters $a_i$ as well) split into $(n-1)$ chains. Attach a specific color to each chain.
Then the coloring of positions and of letters uniquely determines the word $W$. Furthermore the number of these colorings does not exceed
$(n-1)^k\times (n-1)^k=(n-1)^{2k}$.

The above estimate implies validity of polylinear identities corresponding to an irreducible module whose Young diagram includes
the square of size $n^4$. This in turn enables, firstly, to obtain a transparent proof for Regev theorem which asserts that a tensor product of $\PI$-algebras is a $\PI$-algebra as well; secondly, to establish the existence of a sparse identity in the general case and of a Capelli identity in the finitely generated case (and thus to prove the theorem on nilpotency of the radical); and thirdly, to realize A.~R.~Kemer's ``supertrick'' that reduces the study of identities in general algebras to that of super-identities in finitely generated superalgebras of zero characteristic. Close questions are considered in \cite{BP07,Lot83,02}.

Problems related to enumeration of polylinear words which are not $n$-divisible are interesting of their own.
(For example, there exists a bijection between not
$3$-divisible words and Catalana numbers.) On one hand, this is a purely combinatorial problem, but on the other hand, it is related to the set of codimensions for the general matrix algebra. The study of polylinear words seems to be of great importance.
V.~N.~Latyshev  (see for instance \cite{LatyshevMulty}) has stated the problem of finite-basedness of the set of leading polylinear words for a $T$-ideal with respect to taking overwords and to isotonous substitutions. This problem implies the Specht problem for polylinear polynomials and is closely related to the problem of weak Noetherian property for the group algebra of an infinite finitary symmetric group over a field of positive characteristic (for zero charasteristic this was established by
A.~Zalessky). To solve the Latyshev problem, it is necessary to translate properties of
$T$-ideals to the language of polylinear words. In
\cite{BBL97,Belov1} an  attempt was made to realize a project of translation of structure properties of algebras to the language of word combinatorics. Translation to the language of polylinear words is simpler and enables to get some information on words of a general form.

In this paper we transfer V.~N.~Latyshev's technique to the non-polylinear case,
and this enables us to obtain a subexponential estimate in Shirshov height theorem.
G.~R.~Chelnokov suggested the idea of this transfer in 1996.
\subsection{Estimates for height}
The original A.~I.~Shirshov's proof, being purely combinatorial
(it was based on the technique of elimination developed by him for Lie algebras, in particular in the proof of the theorem on freeness), nevertheless implied only primitively recursive estimates. Later A.~T.~Kolotov
\cite{Kolotov} obtained an estimate for $\Ht(A)\leqslant l^{l^n}$\
($n=\deg(A)$,\, $l$~is the number of generators). A.~Ya.~Belov in \cite{Bel92} has shown that $\Ht(n,l)<2nl^{n+1}$. The exponential estimate in Shirshov height theorem was also presented in \cite{BR05, Dr00,Kh11(2),Kh11(3)} . The above estimates were sharpened in the papers by A.~Klein \cite{Klein,Klein1}. In 2001 Ye.~S.~Chibrikov proved in \cite{Ch01} that $\Ht(4,l) \geqslant  (7k^2-2k).$ M.~I.~Kharitonov in \cite{Kh11, Kh11(2),Kh11(3)} obtained estimates for the structure of piecewise periodicity. In 2011 A.~A.~Lopatin \cite{Lop11} obtained the following result:

\begin{theorem}
Let $C_{n,l}$ be the nilpotency degree of a free $l$-generated algebra satisfying $x^n=0$, and let $p$ be the characteristic of the ground field of the algebra, greater than $n/2.$ Then
$$(1): C_{n,l}<4\cdot 2^{n/2} l.$$
\end{theorem}
By definition $C_{n,l}\leqslant \Psi(n, n, l).$
Observe that for small $n$ the estimate (1) is smaller than the estimate $\Psi(n, n, l)$ established in this paper but for growing  $n$ the estimate $\Psi(n,n,l)$ is asymptotically better than (1).

Ye.~I.~Zelmanov has put the following question in the Dniester Notebook \cite{Dnestrovsk} in 1993:
\begin{ques}
Let $F_{2,m}$ be the free $2$-generated associative ring with identity $x^m=0.$ Is it true that the nilpotency class of $F_{2,m}$ grows exponentially in $m?$
\end{ques}

Our paper answers Ye.~I.~Zelmanov's question as follows: the nilpotency class in question grows subexponentially.
\subsection{The results obtained}
{\bf The main result} of the paper is as follows:

\begin{theorem}     \label{c:main2}
The height of the set of not $n$-divisible words over an alphabet of cardinality $l$ relative to the set of words of length less than $n$ does not exceed $\Phi(n,l)$ where
$$\Phi(n,l) = E_1 l\cdot n^{E_2+12\log_3 n} ,$$
$E_1 = 4^{21\log_3 4 + 17}, E_2 = 30\log_3 4 + 10.$
\end{theorem}
This theorem after some coarsening and simplification of the estimate implies that
for fixed $l$ and $n \rightarrow\infty$ we have
$$\Phi(n,l) < 2^{87} l\cdot n^{12\log_3 n + 48}
= n^{12(1+o(1))\log_3{n}},$$
and for fixed $n$ and $l\rightarrow\infty$ we have
$$\Phi(n,l) < C(n)l.$$

\begin{corollary}

The height of an $l$-generated $\PI$-algebra with an admissible polynomial identity of degree
$n$ over the set of words of length less than $n$ does not exceed $\Phi(n,l)$.
\end{corollary}

Moreover we prove a subexponential estimate which is better for small $n$:

\begin{theorem}     \label{t1:log_2}
The height of the set of not $n$-divisible words over an alphabet of cardinality $l$
relative to the set of words of length less than $n$ does not exceed $\Phi(n,l)$ where
$$\Phi(n,l) = 2^{40} l\cdot n^{38+8\log_2 n}.$$
\end{theorem}

In particular we obtain subexponential estimates for the nilpotency index of
$l$-generated nil-algebras of degree $n$ in an  arbitrary characteristic.

The second main result of our paper is the following

\begin{theorem}      \label{c:main1}
Let $l$, $n$ and $d\geqslant n$ be positive integers. Then all
$l$-generated words of length not less than $\Psi(n,d,l)$ either contain
$x^d$ or are $n$-divisible. Here
$$
\Psi(n,d,l)=4^{5+3\log_3 4} l (nd)^{3 \log_3 (nd)+(5+6\log_3 4)}d^2.
$$
\end{theorem}
This theorem after some coarsening and simplification of the estimate implies
that for fixed $l$ and $nd \rightarrow\infty$ we have
$$\Psi(n,d,l) < 2^{18} l (nd)^{3 \log_3 (nd)+13}d^2
= (nd)^{3(1+o(1))\log_3(nd)},$$

and for fixed $n$ and $l\rightarrow\infty$ we have
$$\Psi(n,d,l) < C(n,d)l.$$

\begin{corollary}
Let $l$, $d$ be positive integers, and let an associative $l$-generated algebra $A$ satisfy
$x^{d}=0$. Then its nilpotency index is less than $\Psi(d,d,l)$.
\end{corollary}

Moreover we prove a subexponential estimate which is better for small $n$ and $d$:

\begin{theorem}     \label{t2:log_2}
Let $l$, $n$ and $d\geqslant n$ be positive integers. Then all
$l$-generated words of length not less than $\Psi(n,d,l)$ either contain $x^d$ or are $n$-divisible. Here
$$
\Psi(n,d,l) = 256 l(nd)^{2\log_2 (nd)+10}d^2.$$
\end{theorem}

\begin{notation}
For a real number $x$ put $\ulcorner x\urcorner := -[-x].$ Thus we replace non-integer numbers by the closest greater integers.
\end{notation}
Proving theorem \ref{c:main2} we also prove the following theorem on estimation of essential height:

\begin{theorem} \label{ThThick}
The essential height of an $l$-generated $PI$-algebra with an admissible polynomial identity of degree $n$ over the set of words of length less than $n$ is less than $\Upsilon (n, l),$ where
$$\Upsilon (n, l) = 2n^{3\ulcorner\log_3 n\urcorner + 4} l.$$
\end{theorem}

In \cite{Bog01} it is established that the nilpotency index of an
$l$-generated nil-semiring of degree $n$ equals the nilpotency index of an $l$-generated nilring of degree $n$, where addition is not supposed to be commutative. (The paper also contains examples of non-nilpotent nil-nearrings of index $2$.) Thus our results extend to the case of semirings as well.
\medskip

\subsection{On estimates from below} Let us compare the results obtained with the estimate for height from below. The height of an algebra $A$
is not less than its Gelfand --- Kirillov dimension $\GK(A)$. For the algebra of
$l$-generated general matrices of order $n$ this dimension equals
$(l-1)n^2+1$ (see \cite{Procesi} as well as \cite{Bel04}). At the same time, the minimal degree of an identity in this algebra is
$2n$ by Amitsur --- Levitsky theorem. We have

\begin{proposition}
The height of an $l$-generated $\PI$-algebra of degree $n$ and of the set of
not $n$-divisible words over an alphabet of cardinality $l$ is not less than
$(l-1)n^2/4+1$.
\end{proposition}

Estimates from below for the nilpotency index were established by
Ye.~N.~Kuzmin in \cite{Kuz75}.  He gave an example of
$2$-generated algebra with identity $x^n=0$, such that its nilpotency index exceeds
$(n^2+n-2)/2$. The problem of finding estimates from below is considered in \cite{Kh11}.

At the same time, for zero characteristic and a countable set of generators,
Yu.~P.~Razmyslov (see for instance \cite{Razmyslov3})
obtained an upper estimate for the nilpotency index, namely $n^2$.

First we will prove theorem \ref{c:main1}, and in the following section we will deal with
estimates for essential height, that is, for the number of distinct periodical pieces in a not $n$-divisible word.
\medskip

{\bf Acknowledgements.} The authors are grateful to V.~N.~Latyshev,
A.~V.~Mikhalyov and all participants of the seminar ``Ring theory'' for their attention to our work, as well as to the participants of the seminar at Moscow Physic-technical institute under the guidance of A.~M.~Raigorodsky.

\section{Estimates on occurence of degrees of subwords}
\subsection{The outline of the proof for theorem \ref{c:main1}}

Lemmas \ref{Lm0.1}, \ref{c:lem1.2} and \ref{c:lem1.3} describe sufficient conditions for presence of a period of length $d$ in a not $n$-divisible word $W$. Lemma \ref{c:lem1.4} connects $n$-divisibility of a word $W$ with the set of its tails. Further we choose some specific subset in the set of tails of $W$, such that we can apply Dilworth theorem. After that we color the tails and their first letters according to their location in chains obtained by application of Dilworth theorem.

We have to know the position in any chain where neighboring tails begin to differ. It is of interest what is the $``$frequency$"$ of this position in a $p$-tail for some $p\leqslant n$. Further we somewhat generalize our reasoning dividing tails into segments consisting of several letters each and determining the segment containg the position where neighboring tails begin to differ.
Lemma \ref{c:lem2} connects the $``$frequencies$"$ in question for $p$-tails and $kp$-tails for $k = 3$.

To complete the proof, we construct a hierarchical structure based on lemma \ref{c:lem2}, that is, we consecutively consider segments of $n$-tails, subsegments of these segments and so on. Furthermore we consider the greatest possible number of tails in the subset to which Dilworth theorem is applied, and then we estimate from above the total number of tails and hence of the letters in the word $W$.

\subsection{Periodicity and $n$-divisibility properties}
\smallskip


Let $a_1, a_2,\ldots ,a_l$ be the alphabet used for constructing words. The ordering $a_1\prec a_2\prec\dots\prec a_l$ induces lexicographical ordering for words over the alphabet. For convenience, we introduce the following definitions.

\begin{definition}
a) If a word $v$ includes a subword of the form $u^t$ then we say that $v$ includes a {\it period of length $t.$}

b) If a word $u$ is the beginning of a word $v$ then these words are called {\em incomparable}.

c) A word $v$ is {\em a tail} of a word $u$ if there exists a word $w$ such that $u=wv$.

d) A word $v$ is {\em a $k$-tail} of a word $u$ if $v$ consists of $k$ first letters of some tail $u$.

d*) {\em A $k$-beginning} is the same as $k$-tail.

e) A word $u$ {\em is to the left from} a word $v$ if $u$ begins to the left from the beginning of $v$.
\end{definition}

\begin{notation}
a) For a real number $x$ put $\ulcorner x\urcorner := -[-x].$

b) Let $|u|$ denote the length of a word $u$.
\end{notation}

The proof uses the following sufficient conditions for presence of a period:

\begin{lemma}       \label{Lm0.1}
In a word $W$ of length $x$ either first $[x/d]$ tails are pairwise comparable or
$W$  includes a period of length $d$.
\end{lemma}

\begin{proof}  Suppose $W$ includes no word of the form $u^{d}$. Consider first
$[x/{d}]$ tails. Suppose some two of them, say $v_1$ and $v_2$, are incomparable and
$v_1=u\cdot v_2$. Then $v_2=u\cdot v_3$ for some $v_3$. Furthermore $v_1=u^2\cdot v_3$.
Arguing in this way we obtain that $v_1=u^{d}\cdot
v_{{d}+1}$ since $|u|<x/ {d}$, $|v_2|\geqslant ({d}-1)x/ {d}$.
A contradiction.\end{proof}

\begin{lemma}     \label{c:lem1.2}
If a word $V$ of length $k\cdot t$ includes not more than $k$ different subwords
of length $k$ then $V$ includes a period of length $t$.
\end{lemma}

\begin{proof} We use induction in $k$. The base $k = 1$ is obvious. If there are not more
than $(k - 1)$ different subwords of length $(k - 1)$ then we apply induction assumption. If there exist $k$
different subwords of length $(k - 1)$ then every subword of length $k$
is uniquely determined by its first $(k - 1)$ letters. Thus
$V=v^t$ where $v$ is a $k$-tail of $V$.\end{proof}

\begin{definition}
a) A word $W$ is \textit{$n$-divisible in ordinary sense} if there exist
$u_1, u_2,\ldots,u_n$ such that $W=v\cdot u_1\cdots
u_n$ and $u_1\succ\ldots \succ u_n$.

b) In our proof we will call a word $W$ \textit{$n$-divisible in tail sense} if there exist tails
$u_1,\ldots,u_n$ such that $u_1\succ u_2\succ \ldots \succ u_n$ and for any
$i=1, 2,\ldots, n - 1$ the beginning of $u_i$ is to the left from the beginning of
$u_{i+1}$. If the contrary is not specified, an \textit{$n$-divisible} word means $n$-divisible
in tail sense.

c) A word $W$ is \textit{$n$-cancellable} if either it is
$n$-divisible in ordinary sense or there exists a word of the form $u^{d}\subseteq W$..
\end{definition}

Now we describe a sufficient condition for $n$-cancellability and its connection with $n$-divisibility.

\begin{lemma}          \label{c:lem1.3}
If a word $W$ includes $n$ identical disjoint subwords $u$ of length $n\cdot{d}$ then $W$ is $n$-cancellable.
\end{lemma}

\begin{proof} Suppose the contrary. Consider the tails
$u_1,u_2,\ldots,u_n$ of the word $u$ which begin from each of the first $n$ letters of $u$.
Renumerate the tails to provide the inequalities $u_1\succ\ldots\succ u_n.$ By lemma 1 the tails
are incomparable. Consider the subword $u_1$ in the left-most copy of $u,$ the subword $u_2$ in the second copy from the left, $\ldots, u_n$ in the $n$th copy from the left. We get an
$n$-division of $W$. A contradiction.\end{proof}

\begin{lemma}      \label{c:lem1.4}
If a word $W$ is $4nd$-divisible then it is $n$-cancellable.
\end{lemma}

\begin{proof} Suppose the contrary. Consider the numbers of positions of letters $a_i$,
$a_1<a_2<\ldots<a_{4nd}$, that begin the tails $u_i$ dividing $W$. Set $a_{4nd+1} = |W|$. If $W$ is not $n$-cancellable then there exists $i$, $1\leqslant i\leqslant 4(n-1)d + 1$, such that for any $i\leqslant b<c\leqslant d<e\leqslant i+4d$ the $(a_c - a_b)$-tail $u_b$ is incomparable with the $(a_e-a_d)$-tail $u_d$. Compare $a_{i+2d} - a_i$ and $a_{i+4d} - a_{i+2d}$. We may assume that $a_{i+4d} - a_{i+2d}\geqslant a_{i+2d} - a_i$. Let $a_{j+1} - a_j = \inf\limits_k {(a_{k+1} - a_k)}, 0\leqslant j<2d.$ We may assume that $j<d.$ By assumption the $(a_{2d}-a_j)$-tail $u_j$ and the $(a_{2d} - a_{j+1})$-tail $u_{j+1}$ are incomparable with the $(a_{4d}-a_{2d})$-tail $u_{2d}$. Since $a_{4d}-a_{2d}\geqslant a_{2d}-a_j>a_{2d}-a_{j+1}$, the $(a_2d-a_j)$-tail $u_j$ and the $(a_{2d}-a_{j+1})$-tail $u_{j+1}$ are mutually incomparable. Since ${{a_{2d}-a_j}\over {a_{2d}-a_{j+1}}}\leqslant {{d+1}\over {d}},$ the $(a_{j+1}-a_j)$-tail $u_j$ in degree $d$ is included into the $(a_2d-a_j)$-tail $u_j$.
A contradiction. \end{proof}

\begin{corollary}
If a word $W$ is not $n$-divisible in ordinary sense then $W$ is not
$4nd$-divisible (in tail sense).
\end{corollary}

\begin{notation}
Set $p_{n, d}:=4nd-1$.
\end{notation}


Let $W$ be a not $n$-cancellable word. Consider $U$, the $[\left|
W\right|/d]$-tail of $W$. Then $W$ is not $(p_{n, d}+1)$-divisible.
Let $\Omega$ be the set of tails of $W$, which begin in
$U$. Then by lemma \ref{Lm0.1} any two elements of $\Omega$
are comparable. There is a natural bijection between $\Omega$,
the letters of $U$ and positive integers from $1$ to
$\left|\Omega\right|=\left|U\right|$.

Let us introduce a word $\theta$ which is lexicographically less than any other word.

\begin{remark}
In the current proof of theorem \ref{c:main1} all tails are assumed to belong to $\Omega$.
\end{remark}

\section{Estimates on occurence of periodical fragments}

\paragraph{An application of Dilworth theorem.} For tails $u$ and $v$ put
$u<v$ if $u \prec v$ and $u$ is to the left from $v$. Then
by Dilworth theorem, $\Omega$ can be divided into $p_{n, d}$ chains such that in each chain
$u \prec v$ if $u$ is to the left from $v$. Paint the initial positions of the tails into $p_{n, d}$
colors according to their occurence in chains. Fix a positive integer
$p$. To each positive integer $i$ from 1 to $\left|\Omega\right|$, attach
$B^p(i)$, an ordered set of $p_{n, d}$ words $\{f(i,j)\}$
constructed as follows:

{\it For each $j = 1, 2,\ldots, p_{n, d}$\ put

$$f(i,j)=\left\{\max \
f\leqslant i: f\ \mbox{ is painted into color }j\right\}.$$

If there is no such $f$ then the word from $B^p(i)$ at position
$j$ is assumed to be equal to $\theta$, otherwise to the
$p$-tail that begins from the $f(i,j)$-th letter. }

Informally speaking, we observe the speed of "evolution'' of tails in their chains when the sequence of positions in $W$ is considered as the time axis.

\subsection{The sets $B^p(i)$ and the process at positions}

\begin{lemma}[On the process]   \label{c:lem}
Given a sequence $S$ of length $|S|$ consisting of words of length
$(k-1)$. Each word consists of $(k-2)$ symbols $``0"$ and a single symbol
$``1"$. Let $S$ satisfy the following condition:

{\em If for some $0 < s \leqslant k-1$ there exist $p_{n, d}$
words such that $``1"$ occupies the $s$th position then between the first and the
$p_{n, d}$-th of these words there exists a word such that $``1"$ occupies a position with number strictly less than
$s$}; $L(k-1)=\sup\limits_S |S|$.

Then $L(k-1)\leqslant p_{n, d}^{k-1}-1$.
\end{lemma}

\begin{proof} We have $L(1) \leqslant p_{n, d}-1$. Let $L(k-1) \leqslant {p_{n, d}^{k-1}} -
1$. We will show that $L(k) \leqslant {p_{n, d}^{k}} - 1$.
Consider the words such that $``1"$ occupies the first position.
Their number does not exceed $p_{n, d}-1$. Between any two of them as well as before the first one and after the last one,
the number of words does not exceed $L(k-1)
\leqslant {p_{n, d}^{k-1}} - 1$. Hence
$$
L(k) \leqslant p_{n, d} - 1 +
\left(p_{n, d}\right)\left({\left(p_{n, d}\right)^{k-1}}-1\right) =
{\left(p_{n, d}\right)^{k}}-1,$$
as required.
\end{proof}

We need a quantity which estimates the speed of "evolution'' of sets $B^p(i)$:

\begin{definition}
Set $$\psi(p):= \left\{\max \ k:
B^p(i)=B^p(i+k-1)\right\}.$$
In particular, by lemma \ref{c:lem1.2} we have
$\psi(p_{n,d})\leqslant p_{n, d} d$.

For a given $\alpha$ we divide the sequence of the first $\left|\Omega\right|$ positions
${i}$ of $W$ into equivalence classes $\AC_\alpha$ as follows: $i\AC_\alpha j$ if $B^\alpha(i)=B^\alpha(j).$
\end{definition}

\begin{proposition}
For any positive integers $a<b$ we have $\psi(a) \leqslant \psi(b).$
\end{proposition}

\begin{lemma}[basic] \label{c:lem2}
For any positive integers $a, k$ we have
$$\psi(a)\leqslant p_{n,d}^k\psi(k\cdot a)+k\cdot a$$
\end{lemma}

\begin{proof}
Consider the least representative in each class of  $\AC_{k\cdot a}$. We get a sequence of positions $\{i_j\}.$ Now consider all $i_j$ and $B^{k\cdot a}(i_j)$ from the same equivalence class of $\AC_a.$ Suppose it consists of $B^{k\cdot a}(i_j)$ for $i_j\in[b, c).$ Let $\{i_j\}'$ denote the segment of the sequence $\{i_j\},$ such that $i_j\in[b, c-k\cdot a).$

Fix a positive integer $r, 1\leqslant r\leqslant p_{n,d}.$ All $k\cdot a$-beginnings of color $r$ that begin from positions of the word $W$ in $\{i_j\}'$ will be called representatives of type $r$. All representatives of type $r$ are pairwise distinct because they begin from the least positions in equivalence classes of $\AC_{k\cdot a}.$ Divide each representative of type $r$ into $k$ segments of length $a$. Enumerate segments inside each representative of type $r$ from left to right by integers from zero to $(k-1)$. If there exist $(p_{n,d}+1)$ representatives of type $r$ with the same first $(t-1)$ segments but with pairwise different $t$th segments where $1\leqslant t\leqslant k-1$ then there are two $t$th segments such that their first letters are of the same color. Then the initial positions of these segments belong to different equivalence classes of $\AC_a.$

Now apply lemma \ref{c:lem} as follows: in all representatives of type $r$ except the rightmost one we consider a segment as {\it a unit segment} if it contains the least position where this representative of type $r$ differs from the preceding one. All other segments are considered as {\it zero segments.}

Now we apply the process lemma for the values of parameters as given in the condition of the lemma.  We obtain that the sequence $\{i_j\}'$ contains not more than $p_{n,d}^{k-1}$ representatives of type $r$. Then the sequence $\{i_j\}'$  contains not more than  $p_{n,d}^{k}$ terms. Thus $c-b\leqslant p_{n,d}^k\psi(k\cdot a)+k\cdot a.$
\end{proof}

\subsection{Completion of the proofs for theorems \ref{c:main1} and \ref{t2:log_2}}

Let
$$a_0 = 3^{\ulcorner \log_3 p_{n,d}\urcorner}, a_1 = 3^{\ulcorner \log_3 p_{n,d}\urcorner-1},\ldots,a_{{\ulcorner \log_3 p_{n,d}\urcorner}} =1.$$
Then $\left|W\right|\leqslant d\left|\Omega\right| + d$ by lemma \ref{Lm0.1}.

Since for the set $B^1(i)$ not more than $(1+p_{n,d}l)$ different values are possible, we have $\left|W\right|\leqslant d(1+p_{n,d}l)\psi(1) + d.$

By lemma \ref{c:lem2}
$$\psi(1)< (p_{n,d}^3 + p_{n,d})\psi(3)<(p_{n,d}^3 + p_{n,d})^2\psi(9)<\cdots <(p_{n,d}^3 + p_{n,d})^{\ulcorner \log_3 p_{n,d}\urcorner}\psi(p_{n,d})\leqslant
$$
$$\leqslant (p_{n,d}^3 + p_{n,d})^{\ulcorner \log_3 p_{n,d}\urcorner}p_{n,d}d
$$

Take $p_{n,d} = 4nd-1,$ to get
$$\left|W\right|< 4^{5+3\log_3 4} l (nd)^{3 \log_3 (nd)+(5+6\log_3 4)}d^2.
$$

This implies the assertion of {\bf Theorem \ref{c:main1}.}

The proof of theorem \ref{t2:log_2} is completed similarly but instead of the sequence
$$a_0 = 3^{\ulcorner \log_3 p_{n,d}\urcorner}, a_1 = 3^{\ulcorner \log_3 p_{n,d}\urcorner-1},\ldots,a_{{\ulcorner \log_3 p_{n,d}\urcorner}} =1$$
we have to consider the sequence
$$a_0 = 2^{\ulcorner \log_2 p_{n,d}\urcorner}, a_1 = 2^{\ulcorner \log_2 p_{n,d}\urcorner-1},\ldots,a_{{\ulcorner \log_2 p_{n,d}\urcorner}} =1.$$

\section{An estimate for the essential height.}
In this section we proceed with the proof of the main theorem \ref{c:main2}.
In passing, we prove theorem \ref{ThThick}. We consider positions of letters in the word $W$ as the time axis. That is, a subword $u$ occurs before a subword $v$ if $u$ is entirely to the left from $v$ in $W$.



\subsection{Isolation of distinct periodical fragments in the word~$W$} \label{c:sub1}

Let $s$ denote the number of subwords in $W$ such that each of them includes a period of length less than $n$ more than $2n$ times and each pair of them is separated by subwords of length greater than $n$, comparable with the preceding period. Enumerate these from the beginning to the end of the word: $x_1^{2n}, x_2^{2n},\ldots,x_s^{2n}$. Thus
$W=y_0x_1^{2n}y_1x_2^{2n}\cdots x_s^{2n}y_s.$



If there is $i$ such that the word $x_i$ has length not less than $n$ then the word $x_i^2$ includes $n$ pairwise comparable tails, hence the word $x_i^{2n}$ is $n$-divisible. Then $s$ is not less than the essential height of $W$ over the set of words of length less than $n.$

\begin{definition}
A word $u$ will be called {\em non-cyclic} if $u$ is not representable in the form
$v^k$ where $k>1$.
\end{definition}

\begin{definition}
{\em A word cycle $u$} is the set consisting of the word $u$ and all its cyclis shifts.
\end{definition}

\begin{definition}
A word $W$ is {\em strongly $n$-divisible} if it is representable in the form
$W=W_0W_1\cdots W_n$ where the subwords
$W_1,\dots,W_n$ are placed in the lexicographically decreasing order and each
of $W_i, i=1, 2,\ldots, n$ begins from some word
$z_i^k\in Z$ where all $z_i$ are distinct.
\end{definition}

\begin{lemma}\label{lem4.10}
If there is an integer $m, 1\leqslant m<n,$ such that there exist $(2n-1)$ pairwise incomparable words of length $m: x_{i_1},\ldots ,x_{i_{2n-1}},$ then $W$ is $n$-divisible.
\end{lemma}
\begin{proof} Put $x:=x_{i_1}.$ Then $W$ includes disjoint subwords $x^{p_1}v'_1,\ldots ,x^{p_{2n-1}}v'_{2n-1},$ where $p_1,\ldots ,p_{2n-1}$ are positive integers greater than $n$, and $v'_1,\ldots ,v'_{2n-1}$ are words of length $m$ comparable with $x, v'_1=v_{i_1}.$ Hence among the words $v'_1,\ldots ,v'_{2n-1}$ either there are $n$ words lexicographically greater than $x$ or there are $n$ words lexicographically smaller than $x$. We may assume that $v'_1,\ldots ,v'_n$ are lexicographically greater than $x$. Then $W$ includes subwords $v'_1, xv'_2,\ldots ,x^{n-1}v'_n,$ which lexicographically decrease from left to right.\end{proof}

Consider an integer $m, 1\leqslant m<n.$ Divide all $x_i$ of length $m$ into equivalence classes relative to strong incomparability and choose a single representative from each class. Let these be $x_{i_1},\ldots ,x_{i_s'},$
where $s'$ is a positive integer. Since the subwords $x_i$ are periods, we consider them as word cycles.

\begin{notation}

$v_k := x_{i_k}$

Let $v(k, i)$ where $i$ is a positive integer, $1\le i\le m$, be a cyclic shift of a word $v_k$ by $(k - 1)$ positions to the right, that is, $v(k, 1) = v_k$ and the first letter of $v(k, 2)$ is the second letter of $v_k$. Thus $\{ v(k, i)\}_{i=1}^m$ is a word cycle of $v_k$. Note that for any $1\leqslant i_1, i_2\leqslant p, 1\leqslant j_1, j_2\leqslant m$ the word $v(i_1, j_1)$ is strongly incomparable with $v(i_2, j_2)$.
\end{notation}


\begin{remark}
The cases $m = 2, 3, n-1$ were considered in \cite{Kh11, Kh11(2),Kh11(3)}.
\end{remark}

\subsection{An application of Dilworth theorem}      \label{c:sub2}

Consider a set $\Omega' = \{ v(i, j)\}$, where $1\leqslant i\leqslant p, 1\leqslant j\leqslant m.$ Order the words $v(i, j)$ as follows:

$v(i_1, j_1)\succ v(i_2, j_2)$ if

1) $v(i_1, j_1)> v(i_2, j_2)$

2) $i_1 > i_2$.

\begin{lemma}\label{c:lem4}
If in the set $\Omega'$ with ordering $\succ$ there exists an antichain of length $n$ then $W$ is $n$-divisible.
\end{lemma}
\begin{proof} Suppose there exists an antichain of length $n$ consisting of words\\ $v(i_1, j_1), v(i_2, j_2),\ldots,v(i_n, j_n)$; here $i_1\leqslant i_2\leqslant\cdots\leqslant i_n.$ If all inequalities between $i_k$ are strict then $W$ is $n$-divisible by definition.

Suppose that for some $r$ there exist $i_{r+1} = \cdots = i_{r+k}$ such that either $r = 0$ or $i_r < i_{r+1}$. Moreover the positive integer $k$ is such that either $k = n - r$ or $i_{r+k} < i_{r+k+1}$.

The word $s_{i_{r+1}}$ is periodical, hence it is representable as a product of $n$ copies of $v^2_{i_{r+1}}$. The word $v^2_{i_{r+1}}$ includes a word cycle $v_{i_{r+1}}$. Hence in $s_{i_{r+1}}$ there exist disjoint subwords placed in lexicographically decreasing order and equal to $v(i_{r+1}, j_{r+1}),\ldots,v(i_{r+k}, j_{r+k})$ respectively. Similarly we deal with all sets of equal indices in the sequence $\{i_r\}_{r=1}^n$. The result is $n$-divisibility of $W$. A contradiction. \end{proof}

Thus $\Omega'$ can be divided into $(n-1)$ chains.

\begin{notation}
Put $q_n = (n-1)$.
\end{notation}

\subsection{The sets $C^\alpha(i)$, the process at positions}

Paint the first letters of the words from $\Omega'$ into $q_n$ colors according to their occurence in chains.
Paint also the integers from $1$ to $\left|\Omega '
\right|$ into the corresponding colors. Fix a positive integer
$\alpha\leqslant m$. To each integer $i$ from 1 to $\left|\Omega '
\right|$, attach an ordered set $C^\alpha(i)$ of $q_n$ words in the following way:

\medskip
{\it For each $j=1,2,\ldots,q_n$\ put
$f(i,j)=\{\max \ f\leqslant i:$ there exists $k$ such that $v(f,k)$ is painted into color $j$ and
the $\alpha$-tail beginning from $f$ consists only of letters initial in some tails from $\Omega '\}$.

If  there is no such $f$ then a word from $C^\alpha(i)$ is assumed to be equal to
$\theta$, otherwise we assume it to be equal to the $\alpha$-tail of $v(f,k)$. }

\begin{notation}
Set $\phi(a)=\{\max\ k:$ for some $i$ we have $C^a(i)=C^a(i+k-1)\}$.

For a given $a\leqslant m$ define a division of the sequence of word cycles
$\{i\}$ in $W$ into equivalence classes as follows: $i \AC_a j$ if $C^a(i) = C^a(j)$.
\end{notation}

Note that the above construction is rather similar to the construction from the proof of theorem \ref{c:main1}. Observe that $B^a(i)$ and $C^a(i)$ are rather similar as well as $\psi(a)$ and $\phi(a)$.

\begin{lemma}\label{lem:m}
$\phi(m) \leqslant q_n/m$.
\end{lemma}

\begin{proof}

In notation 4.1 we have enumerated word cycles. Consider the word cycles with numbers $i, i+1,\ldots, i+[q_n/m].$ We have shown that each word cycle consists of $m$ distinct words. Now consider words in the word cycles $i, i+1,\ldots, i+[q_n/m]$ as elements of the set $\Omega'.$ Then the first letter in each word cycle gets some position. The total number of the positions in question is not less than $n.$ Hence at least two of these positions are of the same color. Now strong incomparability of word cycles implies the assertion of the lemma.
\end{proof}

\begin{proposition}
For any positive integers $a<b$ we have $\phi(a) \leqslant \phi(b).$
\end{proposition}

\begin{lemma}[Basic] \label{lem:thick}
For positive integers $a, k$ such that $ak\leqslant m$ we have
$$\phi(a)\leqslant p_{n,d}^k\phi(k\cdot a)$$
\end{lemma}

\begin{proof}
Consider the minimal representative in each class of $\AC_{k\cdot a}$. We get a sequence of positions $\{i_j\}.$ Now consider all $i_j$ and $C^{k\cdot a}(i_j)$ from the same equivalence class of $\AC_a.$ Suppose it consists of $C^{k\cdot a}(i_j)$ for $i_j\in[b, c).$ Let $\{i_j\}'$ denote the segment of the sequence $\{i_j\},$ such that $i_j\in[b, c).$

Fix a positive integer $r$, $1\leqslant r\leqslant q_n.$ All $k\cdot a$-beginnings of color $r$ that begin from positions of $W$ in $\{i_j\}'$ will be called representatives of type $r$. All representatives of type $r$ are distinct because they begin at the least positions in equivalence classes of $\AC_{k\cdot a}.$ Divide each representative of type $r$ into $k$ segments of length $a$. Enumerate the segments of each representative of type $r$ from left to right by integers from zero to $(k-1)$.
If there exist $(q_n+1)$ representatives of type $r$ with the same first $(t-1)$ segments but pairwise different $t$th segments where $1\leqslant t\leqslant k-1$ then there are two $t$th segments such that their first letters are of the same color. Then the initial positions of these segments belong to different equivalence classes of $\AC_a.$

Now apply lemma \ref{c:lem} in the following way: in all representatives of type $r$ except the rightmost one we consider a segment as {\it a unit segment} if it contains the least position where this representative of type $r$ differs from the preceding one. All other segments are considered as {\it zero segments.}

Now we can apply the process lemma for the values of parameters as given in the condition of the lemma.  We obtain that the sequence $\{i_j\}'$ contains not more than $q_n^{k-1}$ representatives of type $r$. Then the sequence $\{i_j\}'$  contains not more than  $q_n^{k}$ terms. Thus $c-b\leqslant q_n^k\phi(k\cdot a).$
\end{proof}

\subsection{Completion of the proof for theorem \ref{ThThick}}\label{ThThick:end}

Suppose
$$a_0 = 3^{\ulcorner \log_3 p_{n,d}\urcorner}, a_1 = 3^{\ulcorner \log_3 p_{n,d}\urcorner-1},\ldots,a_{{\ulcorner \log_3 p_{n,d}\urcorner}} =1.$$
Substitute these $a_i$ into lemmas
\ref{lem:thick} and \ref{lem:m} to obtain
$$\phi(1)\leqslant q_n^3\phi(3)\leqslant q_n^9\phi(9)\leqslant\cdots \leqslant q_n^{3\ulcorner \log_3 m\urcorner}\phi(m)\leqslant
$$
$$\leqslant q_n^{3\ulcorner \log_3 m\urcorner+1}. $$

Since $C_i^{1}$ takes not more than $1+q_n l$ distinct values, we have

$$\left|\Omega'\right|< q_n^{3\ulcorner \log_3 m\urcorner+1}(1+q_n l)< n^{3\ulcorner \log_3 n\urcorner+2} l.$$

By virtue of lemma \ref{lem4.10} the number of subwords $x_i$ of length $m$ is less than $2n^{3\ulcorner \log_3 n\urcorner+3} l.$
Thus the total number of subwords $x_i$ is less than $2n^{3\ulcorner \log_3 n\urcorner+4} l.$
So $s<2n^{3\ulcorner \log_3 n\urcorner+4} l$ and theorem \ref{ThThick} is proven.

\section{Proof of the main theorem~\ref{c:main2} and of theorem \ref{t1:log_2} }

\subsection{Outline of the proof}
Now {\em an $n$-divisible word} will mean a word $n$-divisible in ordinary sense. To start with, we find the necessary number of fragments in $W$ with length of the period not less than $2n$. For this, it suffices to divide $W$ into subwords of large length and to apply theorem \ref{c:main1} to them. However the estimate can be improved. For this, we find a periodical fragment $u_1$ in $W$ with the period length not less than $4n$. Removing $u_1$, we obtain a word $W_1$. In $W_1$ we find a fragment $u_2$ with period length not less than $4n$ and remove it to get a word $W_2$.
Now we again remove a periodical fragment and proceed in this way, as is described in the algorithm \ref{c:al} in more detail. Then we restore the original word $W$ using the removed fragments. Further we show that a subword $u_i$ in $W$ usually is not a product of a big number of not neighboring subwords. In lemma \ref{c:lem3} we prove that application of algorithm \ref{c:al} enables to find the necessary number of removed subwords of $W$ with period length not less than $2n$.

\subsection{Summing up of essential heights and nilpotency degrees}

\begin{notation}
Let $\Ht(w)$ denote the height of a word $w$ over the set of words of degree not exceeding $n$.
\end{notation}

Consider a word $W$ of height $\Ht(W)>\Phi(n,l)$. Apply the following algorithm to it:

\begin{algorithm}  \label{c:al}
{\ }\\

\begin{description}
    \item[Step $1.$] By theorem \ref{c:main1} the word $W$  includes a subword with period length
    $4n$. Suppose $W_0=W=u'_1x_{1'}^{4n}y'_1$ where the word $x_{1'}$ is not cyclic.
    Represent $y'_1$ in the form $y'_1=x_{1'}^{r_2}y_1$ where $r_2$ is maximal possible.
    Represent $u'_1$ as $u'_1=u_1x_{1'}^{r_1}$ where $r_1$ is maximal possible. Denote by $f_1$ the word
$$W_0=u_1x_{1'}^{4n+r_1+r_2}y_1=u_1f_1y_1.$$
    In the sequel, the positions contained in $f_1$ are called {\em tedious}, the last position of $u_1$ is called {\em tedious of type $1$,} the second position from the end in $u_1$ is called {\em tedious of type $2$,} ..., the $n$th position from the end in $u_1$ is called {\em tedious of type $n$.} Put $W_1 = u_1 y_1.$

    \item[Step $k.$] Consider the words $u_{k-1},\ y_{k-1},\ W_{k-1}=u_{k-1}y_{k-1}$
    constructed at the preceding step. If $|W_{k-1}|\geqslant\Phi(n,l)$ then we apply theorem \ref{c:main1} to $W$ with the restriction that the process in the main lemma \ref{c:lem2} is applied only to non-tedious positions and to tedious positions of type greater than $ka$ where $k$ and $a$ are the parameters from lemma \ref{c:lem2}.

    Thus $W_{k-1}$ includes a non-cyclic subword with period length $4n$ such that
$$W_{k-1}=u'_kx_{k'}^{4n}y'_k.$$
Then put
$$r_1 := \sup \{r: u'_k = u_k x_{k'}^r\},\quad r_2 :=\sup
\{r: y'_k = x_{k'}^r y_k\}.$$
(Note that the words involved may be empty.)\\
Define $f_k$ by the equation
$$W_{k-1}= u_kx_{k'}^{4n+r_1+r_2}y_k = u_kf_ky_k.$$
In the sequel, the positions contained in $f_k$ are called {\em tedious}, the last position in $u_k$ is {\em tedious of type $1$,} the second position from the end in $u_k$ is {\em tedious of type $2$,} ..., the $n$th position from the end in $u_k$ is {\em tedious of type $n$}. If a position occurs to be tedious of two types then the lesser type is chosen to it. Put $W_k = u_k y_k.$
\end{description}

\end{algorithm}

\begin{notation}
Perform $4t+1$ steps of the algorithm \ref{c:al} and consider the original word $W$. For each integer $i$ from the segment $[1,4t]$ we have
$$
W = w_0f_i^{(1)}w_1f_i^{(2)}\cdots f_i^{(n_i)} w_{n_i}
$$
for some subwords $w_j$. Here $f_i = f_i^{(1)}\cdots
f_i^{(n_i)}$. Moreover we assume that for $1\leqslant j\leqslant
n_i-1$ the subword $w_j$ is not empty. Let $s(k)$ be the number
of indices $i\in[1,4t]$ such that $n_i = k$.
\end{notation}

To prove theorem \ref{c:main1}, we have to find as many long periodic fragments as possible. For this, we can use the following lemma:

\begin{lemma}   \label{c:lem3}
$s = s(1) + s(2) \geqslant 2t$.
\end{lemma}

\begin{proof} A subword $U$ of the word $W$ will be called {\it monolithic}  if
\begin{enumerate}
    \item $U$ is a product of words of the form $f_i^{(j)}$,
    \item $U$ is not a proper subword of a word which satisfies the above condition (1).
\end{enumerate}

Suppose that after the $(i-1)$th step of the algorithm \ref{c:al} the word $W$
includes $k_{i-1}$ monolithic subwords. Note that $k_i
\leqslant k_{i-1}-n_i+2$.

 Then if $n_i \geqslant 3$ then $k_i\leqslant k_{i-1}-1.$
If $n_i\leqslant 2$ then $k_i\leqslant k_{i-1}+1$. Furthermore
$k_1=1$, $k_t \geqslant 1 = k_1$. The lemma is proven. \end{proof}

\begin{corollary} \label{c:cor}
$$\sum\limits_{k=1}^\infty {k\cdot s(k)} \leqslant 10t\leqslant 5s.(\ref{c:cor})$$
\end{corollary}

\begin{proof} From the proof of lemma \ref{c:lem3} we obtain
$\sum\limits_{n_i\geqslant 3} {(n_i-2)} \leqslant 2t$.

By definition
$\sum\limits_{k=1}^\infty {s(k)} =4t$, т.е.
$\sum\limits_{k=1}^\infty {2s(k)} =8t$.

Summing up these two inequalities and applying lemma \ref{c:lem3} we obtain the required inequality
 \ref{c:cor}.\end{proof}

\begin{proposition}
The height of $W$ does not exceed $$\Psi(n,4n,l)+\sum\limits_{k=1}^\infty {k\cdot s(k)}\leqslant \Psi(n,4n,l)+ 5s.$$
\end{proposition}

In the sequel, we consider only $f_i$ with $n_i \leqslant 2$.

\begin{notation}
If $n_i = 1$ then put $f'_i := f_i$.

If $n_i = 2$ then put $f'_i := f_i^{(j)}$ where $f_i^{(j)}$
is the word of maximal length between $f_i^{(1)}$ and $f_i^{(2)}$.

Order the words $f'_i$ according to their distance from the beginning of $W$.
We get a sequence $f'_{m_1},\ldots ,f'_{m_s}$ where
$s'=s(1)+s(2)$. Put $f''_i := f'_{m_i}$. Suppose $f''_i = w'_i
x_{i''}^{p_{i''}}w''_i$ where at least one of the words $w'_i, w''_i$ is empty.
\end{notation}

\begin{remark}  \label{c:pr}
We may assume that at starting steps of algorithm \ref{c:al}
we have chosen all $f_i$ such that $n_i =1$.
\end{remark}

Now consider $z'_j$, the subwords in $W$
of the following form:
$$z'_j = x_{(2j-1)''}^{p_{(2j-1)''}+\gimel}v_j,
\gimel\geqslant 0, |v_j| = |x_{(2j-1)''}|,$$
here $v_j$ is not equal to $x_{(2j - 1)''}$, and  the beginning of $z'_j$ coincides with the beginning of a periodic subword in $f''_{2j-1}$. We will show that $z'_j$ are disjoint.

Indeed, if $f''_{2j-1} = f_{m_{2j-1}}$ then
$z'_j=f_{m_{2j-1}}v_j$.

If $f''_{2j-1}= f_{m_{2j-1}}^{(k)}$, $k = 1,2$, and
$z'_j$ intersects $z'_{j+1}$ then $f''_{2j}\subset z'_i$. Since
$x_{(2j)''}$ and $x_{(2j-1)''}$ are not cyclic, we have
$|x_{(2j)''}| = |x_{(2j-1)''}|$. But then the period length in $z'_j$
is not less than $4n$, a contradiction with remark \ref{c:pr}.

Thus we have proved the following lemma:
\begin{lemma}\label{lThick}
In a word $W$ with height not greater than $(\Psi(n,4n,l)+ 5s')$ there exist at least $s'$ disjoint periodic subwords such that the period occurs in each of them at least $2n$ times. Furthermore between any two elements of this set of periodic subwords there is a subword with the same period length as the leftmost of these two elements.
\end{lemma}

\subsection{Completion of the proof for the main theorem~\ref{c:main2} and for theorem \ref{t1:log_2}}
Replace $s'$ in lemma \ref{lThick} by $s$ from the proof of theorem \ref{ThThick} to obtain that
 the height of $W$ does not exceed
$$\Psi(n,4n,l)+ 5s < E_1 l\cdot n^{E_2+12\log_3 n} ,$$ where
$E_1 = 4^{21\log_3 4 + 17}, E_2 = 30\log_3 4 + 10.$

Thus we have obtained {\bf the assertion of the main theorem \ref{c:main2}}.

{\bf Proof of theorem \ref{t1:log_2}} is completed similarly but we have to replace in part \ref{ThThick:end} the sequence
$$a_0 = 3^{\ulcorner \log_3 p_{n,d}\urcorner}, a_1 = 3^{\ulcorner \log_3 p_{n,d}\urcorner-1},\ldots,a_{{\ulcorner \log_3 p_{n,d}\urcorner}} =1$$
by the sequence
$$a_0 = 2^{\ulcorner \log_2 p_{n,d}\urcorner}, a_1 = 2^{\ulcorner \log_2 p_{n,d}\urcorner-1},\ldots,a_{{\ulcorner \log_2 p_{n,d}\urcorner}} =1,$$
and to take the value of $\Psi(n,4n,l)$ from theorem \ref{t2:log_2}.

\section{Comments}

The technique presented to the reader appears to enable improvement of the estimate obtained in this paper.
However this estimate will remain subexponential. A polynomial estimate if it exists, requires new ideas and methods.

At the beginning of the solution presented, subwords of a large word in the application of Shirshov theorem are used mainly as a set of independent elements, not as a set of closely related words. Further we use a coloring of letters inside subwords.
Account of coloring of first letters only leads to an exponential estimate. Account of coloring of all letters in the subwords results in an exponent as well. This fact is due to constructing of a hierarchical system of subwords.  A detailed investigation of the presented connection between subwords together with the solution presented above may improve the presented estimate up to a polynomial one.

It is also of interest to obtain estimates for height of an algebra over the set of words whose degrees do not exceed the complexity of the algebra ($\PI$-degree in literature in English). The paper
\cite{BBL97} presents exponential estimates, and for words that are not a linear combination of lexicographically smaller words, overexponential estimates were obtained in \cite{Bel07}.

The deep ideas of original works by A.~I.~Shirshov
\cite{Shirshov1,Shirshov2} which stem from elimination technique in Lie algebras, may be highly useful, among other issues, for improvement of estimates, despite the fact that the estimates for height in these papers are only primitively recursive.


\end{document}